\documentclass[11pt]{article}

\usepackage{amsfonts,amsmath,amsthm,amssymb}
\usepackage{bm}
\usepackage[colorlinks, citecolor=blue, urlcolor=blue]{hyperref}
\usepackage[numbers]{natbib}
\usepackage{geometry}
\usepackage[utf8]{inputenc}
\usepackage[T1]{fontenc}    
\usepackage{mathtools}

\usepackage{enumitem}

\newtheorem{theorem}{Theorem}[section]

\newtheorem{lemma}[theorem]{Lemma}
\newtheorem{proposition}[theorem]{Proposition}

\theoremstyle{definition}
\newtheorem{definition}[theorem]{Definition}

\newtheorem{remark}[theorem]{\textbf{Remark}}
\numberwithin{equation}{section}

\begin{document}

\title{Markovian projections for It\^o semimartingales with jumps}
\author{
Martin Larsson\footnote{Department of Mathematical Sciences, Carnegie Mellon University, \texttt{larsson@cmu.edu}} \and
Shukun Long\footnote{Department of Mathematical Sciences, Carnegie Mellon University, \texttt{shukunl@andrew.cmu.edu}}
}
\maketitle

\begin{abstract}
Given a general It\^o semimartingale, its Markovian projection is an It\^o process, with Markovian differential characteristics, that matches the one-dimensional marginal laws of the original process. We construct Markovian projections for It\^o semimartingales with jumps, whose flows of one-dimensional marginal laws are solutions to non-local Fokker--Planck--Kolmogorov equations (FPKEs). As an application, we show how Markovian projections appear in building calibrated diffusion/jump models with both local and stochastic features.
\end{abstract}

\section{Introduction}
The Markovian projection arises in the problem where we want to mimic the one-dimensional marginal laws of an It\^o process using another one with simpler dynamics. To be more specific, suppose we are given an It\^o process $X$ whose characteristics are general stochastic processes. Our goal is to find another It\^o process $\widehat{X}$ solving a Markovian SDE, i.e.\ the coefficients are functions of time and the process itself, such that the law of $\widehat{X}_t$ agrees with the law of $X_t$ for every $t \geq 0$. The process $\widehat{X}$ is called a Markovian projection of $X$.

The terminology \emph{Markovian projection} has no standard definition, but is widely used in literature. Some authors require the mimicking process $\widehat{X}$ to be a true Markov process, while others (including our paper) only require $\widehat{X}$ to solve a Markovian SDE and we know the Markov property is not guaranteed in general. Also, some authors prefer to use alternative terminologies like ``mimicking process'' or ``mimicking theorem'' when referring to the same problem.

The idea of Markovian projections for It\^o processes originated from Gy\"{o}ngy \cite{MR0833267}, which was inspired by Krylov \cite{MR0808203}. In \cite{MR0833267} Markovian projections were constructed for continuous It\^o semimartingales, under some boundedness and non-degeneracy conditions on the coefficients. Brunick and Shreve \cite{MR3098443} extended the results of \cite{MR0833267} by relaxing the assumptions therein to an integrability condition. They also proved mimicking theorems for functionals of sample paths such as running average and running maximum, using techniques of updating functions. Bentata and Cont \cite{bentata2009mimicking} studied Markovian projections for It\^o semimartingales with jumps. Their proof was based on a uniqueness result of the FPKE, and the mimicking process they constructed was Markov. To get such results, they imposed relatively strong assumptions on the coefficients such as continuity, which is not always easy to check in practice. See also K\"{o}pfer and R\"{u}schendorf \cite{MR4594213} for work closely related to \cite{bentata2009mimicking}.

In this paper, we construct Markovian projections for c\`adl\`ag It\^o semimartingales. Our results holds under reasonable integrability and growth conditions. In the context of mimicking marginal laws of the process itself, this paper complements Brunick and Shreve \cite{MR3098443} by allowing the diffusion process to have jumps. On the other hand, we work under different settings from Bentata and Cont \cite{bentata2009mimicking}. Our assumptions are weaker in most cases, at the cost of not guaranteeing the uniqueness and Markov property of the mimicking process. One of our main tools is the superposition principle established by R\"{o}ckner, Xie and Zhang \cite{MR4168386}, which constitutes a bridge from weak solutions of non-local FPKEs to martingale solutions for the associated non-local operator. The idea of using a superposition principle to prove a mimicking theorem seems to have been first used in Lacker, Shkolnikov and Zhang \cite{MR4612111}.

This paper is organized as follows. In Section~\ref{sec2} we gather all the required preliminaries. In Section~\ref{sec3} we state and prove our main result (Theorem~\ref{thm_mp}). In Section~\ref{sec4} we provide several examples to illustrate how the theorem can be applied.

Throughout this paper, we let $(\Omega, \mathcal{F}, (\mathcal{F}_t)_{t \geq 0}, \mathbb{P})$ be a filtered probability space satisfying the usual conditions, and we use the following notation:
\begin{itemize}[nosep]
\item $\mathbb{R}_+ = [0, \infty)$.

\item $\mathbb{S}_+^d$ is the set of symmetric positive semi-definite $d \times d$ real matrices.

\item $C_0(\mathbb{R}^d)$ (resp.\ $C_c(\mathbb{R}^d)$) is the set of continuous functions on $\mathbb{R}^d$ which vanish at infinity (resp.\ have compact support).

\item $\mu(f) = \int f \,d\mu$, for $\mu$ a measure and $f$ a measurable function on some space such that the integral is well-defined.

\item $\mathcal{P}(X)$ is the space of Borel probability measures on a Polish space $X$, endowed with the topology of weak convergence.
\end{itemize}

\section{Prerequisites and Preliminary Results}\label{sec2}
This section serves as a preparation for stating and proving our main results. In the sequel, we review some standard notions and present two key lemmas.

\subsection{Transition Kernel}
In the study of the characteristics of It\^o semimartingales with jumps (see Section~\ref{sec_dc}) and other fields like Markov processes, the notion of transition kernels comes into play. In this subsection, we recall some of the standard definitions and fix some terminologies for our later use.

\begin{definition}[Transition kernel]
Let $(X, \mathcal{A})$, $(Y, \mathcal{B})$ be two measurable spaces. We call $\kappa: X \times \mathcal{B} \to [0, \infty]$ a \emph{transition kernel} from $(X, \mathcal{A})$ to $(Y, \mathcal{B})$ if:
\begin{enumerate}[label=(\roman*), nosep]
\item for each $x \in X$, the map $\kappa(x, \cdot): \mathcal{B} \to [0, \infty]$ is a measure,

\item for each $B \in \mathcal{B}$, the map $\kappa(\cdot, B): X \to [0, \infty]$ is a measurable function.
\end{enumerate}
\end{definition}

We often say $\kappa$ is a transition kernel from $X$ to $Y$ if there is no ambiguity on the $\sigma$-algebras $\mathcal{A}$, $\mathcal{B}$. Unless otherwise specified, on a topological space we consider its Borel $\sigma$-algebra; on a product space we consider its product $\sigma$-algebra. In particular, when working with stochastic processes, we assume by default that $\Omega \times \mathbb{R}_+$ is equipped with the $\sigma$-algebra $\mathcal{F} \times \mathcal{B}(\mathbb{R}_+)$. If we require stronger measurability, e.g.\ with respect to the predictable $\sigma$-algebra, we will explicitly say so.

When $X = \Omega$, we also call $\kappa$ a \emph{random measure}. We often use the notation $\kappa(dy)$, omitting its dependency on $\omega \in \Omega$. When $X = \Omega \times \mathbb{R}_+$, for fixed $t \geq 0$ the map
\begin{equation*}
	\Omega \times \mathcal{B} \ni (\omega, B) \mapsto \kappa(\omega, t, B) \in [0, \infty]
\end{equation*}
is a random measure, and we denote it by $\kappa_t(dy)$.

The following terminologies will be convenient for our later use.

\begin{definition}
Let $(X, \mathcal{A})$, $(Y, \mathcal{B})$ be two measurable spaces, and $\kappa$ be a transition kernel from $X$ to $Y$.
\begin{enumerate}[label=(\roman*), nosep]
\item We say $\kappa$ is a \emph{finite} transition kernel if for each $x \in X$, $\kappa(x, dy)$ is a finite measure on $Y$.

\item When $Y = \mathbb{R}^d$, we say $\kappa$ is a \emph{L\'evy} transition kernel if for each $x \in X$, $\kappa(x, dy)$ is a L\'evy measure on $\mathbb{R}^d$, i.e.\
\begin{equation*}
	\kappa(x, \{0\}) = 0
	\quad\text{and}\quad
	\int_{\mathbb{R}^d} 1 \land |y|^2 \,\kappa(x, dy) < \infty.
\end{equation*}

\item When $X = \Omega \times \mathbb{R}_+$ and $\mathcal{A}$ is the predictable $\sigma$-algebra, we say $\kappa$ is a \emph{predictable} transition kernel. That is, for each $B \in \mathcal{B}$, $(\omega, t) \mapsto \kappa(\omega, t, B)$ is a predictable process.
\end{enumerate}
\end{definition}

\subsection{Key Lemmas}
Now we present two lemmas which are crucial in proving our main results. These lemmas are also of interest on their own. The first lemma was proved by Brunick and Shreve \cite{MR3098443}, which we quote below.

\begin{lemma}[cf.\ \cite{MR3098443}, Proposition 5.1]\label{lem1}
Let $X$ be an $\mathbb{R}^d$-valued measurable process, and $\alpha$ be a $C$-valued measurable process, where $C \subseteq \mathbb{R}^n$ is a closed convex set, satisfying
\begin{equation*}
	\mathbb{E}\bigg[\int_0^t |\alpha_s| \,ds\biggr] < \infty,\quad
	\forall\, t > 0.
\end{equation*}
Then, there exists a measurable function $a: \mathbb{R}_+ \times \mathbb{R}^d \to C$ such that for Lebesgue-a.e.\ $t \geq 0$,
\begin{equation*}
	a(t, X_t) = \mathbb{E}[\alpha_t \,|\, X_t].
\end{equation*}
\end{lemma}

\begin{remark}\label{rem_lem1}
For each fixed $t \geq 0$, we all know $\mathbb{E}[\alpha_t \,|\, X_t]$ is some measurable function of $X_t$. However, the joint measurability of $a$ is less obvious, and this is the key point of Lemma~\ref{lem1}. The proof of this lemma is constructive. Indeed, we define the $\sigma$-finite measure $\mu$ and the $\sigma$-finite vector-valued measure $\nu$ via
\begin{equation}\label{def_mu}
\begin{split}
	\mu(A) &\coloneqq \mathbb{E}\biggl[\int_0^\infty \bm{1}_A(s, X_s) \,ds \biggr],\quad
	A \in \mathcal{B}(\mathbb{R}_+ \times \mathbb{R}^d),\\
	\nu(A) &\coloneqq \mathbb{E}\biggl[\int_0^\infty \alpha_s \bm{1}_A(s, X_s) \,ds \biggr],\quad
	A \in \mathcal{B}(\mathbb{R}_+ \times \mathbb{R}^d).
\end{split}
\end{equation}
Clearly, we have $\nu \ll \mu$. Then, we can choose function $a$ to be any version of the Radon--Nikodym derivative $\frac{d\nu}{d\mu}$. For more details, see the proof in \cite{MR3098443}.
\end{remark}

The second lemma is novel, and it is an analogue of Lemma~\ref{lem1} in terms of transition kernels. We will construct a kernel $k(t, x, d\xi)$ from $\mathbb{R}_+ \times \mathbb{R}^d$ to $\mathbb{R}^d$ satisfying an identity involving conditional expectations. The key point is to find a family of measures indexed by $(t, x)$, and simultaneously preserve the joint measurability in $(t, x)$.

\begin{lemma}\label{lem2}
Let $X$ be an $\mathbb{R}^d$-valued measurable process, and $\kappa$ be a transition kernel from $\Omega \times \mathbb{R}_+$ to $\mathbb{R}^d$ satisfying
\begin{equation}\label{lem2_asm}
	\mathbb{E}\bigg[\int_0^t \kappa_s(\mathbb{R}^d) \,ds\biggr] < \infty,\quad
	\forall\, t > 0.
\end{equation}
Then, there exists a finite transition kernel $k$ from $\mathbb{R}_+ \times \mathbb{R}^d$ to $\mathbb{R}^d$ such that for Lebesgue-a.e.\ $t \geq 0$,
\begin{equation}\label{lem2_rslt}
	k(t, X_t, A) = \mathbb{E}[\kappa_t(A) \,|\, X_t],\quad
	\forall\, A \in \mathcal{B}(\mathbb{R}^d).
\end{equation}
\end{lemma}

\begin{proof}
By the integrability condition \eqref{lem2_asm}, without loss of generality, we may assume that $\kappa$ is a finite transition kernel. Otherwise, we can simply modify $\kappa(\cdot, \cdot, d\xi) \coloneqq 0$ on a $(\mathbb{P} \otimes dt)$-null set.

Our proof is based on the Riesz--Markov--Kakutani representation theorem for the dual space of $C_0(\mathbb{R}^d)$. Since nonzero constant functions do not belong to $C_0(\mathbb{R}^d)$, for technical reasons we first consider the function space
\begin{equation*}
	C_\ell(\mathbb{R}^d)
	\coloneqq C_0(\mathbb{R}^d) \oplus \mathbb{R}
	= \{f + c: f \in C_0(\mathbb{R}^d),\, c \in \mathbb{R}\}.
\end{equation*}
In other words, $C_\ell(\mathbb{R}^d)$ is the space of continuous functions on $\mathbb{R}^d$ which admit a finite limit at infinity. We endow $C_\ell(\mathbb{R}^d)$ with the supremum norm. Since $C_0(\mathbb{R}^d)$ is a separable Banach space, it is easy to check that $C_\ell(\mathbb{R}^d)$ is also a separable Banach space. Let $\mathcal{C}$ be a countable dense subset of $C_\ell(\mathbb{R}^d)$ with $1 \in \mathcal{C}$. Let $\mathcal{L}$ be the $\mathbb{Q}$-span of $\mathcal{C}$, i.e.\ the collection of all finite linear combinations of elements of $\mathcal{C}$ with rational coefficients. Clearly, $\mathcal{L}$ is a countable dense subset of $C_\ell(\mathbb{R}^d)$ with $1 \in \mathcal{L}$. Moreover, $\mathcal{L}$ is a vector space over $\mathbb{Q}$ by construction.

For each $\varphi \in \mathcal{L}$, by \eqref{lem2_asm} and Lemma~\ref{lem1}, there exists an $\mathbb{R}$-valued measurable function of $(t, x) \in \mathbb{R}_+ \times \mathbb{R}^d$, denoted by $L_{t, x}(\varphi)$, such that for Lebesgue-a.e.\ $t \geq 0$,
\begin{equation}\label{lem2_L}
	L_{t, X_t}(\varphi) = \mathbb{E}\biggl[\int_{\mathbb{R}^d} \varphi(\xi) \,\kappa_t(d\xi) \,\bigg|\, X_t\biggr].
\end{equation}
Now for fixed $(t, x) \in \mathbb{R}_+ \times \mathbb{R}^d$, we can view $\varphi \mapsto L_{t, x}(\varphi)$ as a functional on $\mathcal{L}$. We expect $L_{t, x}$ to be a positive $\mathbb{Q}$-linear functional, but this is not guaranteed unless for each $\varphi \in \mathcal{L}$ we carefully modify the function $(t, x) \mapsto L_{t, x}(\varphi)$.

As discussed in Remark~\ref{rem_lem1}, $(t, x) \mapsto L_{t, x}(\varphi)$ is defined via the Radon--Nikodym derivative $\frac{d\nu_\varphi}{d\mu}$, where $\mu$ is as defined in \eqref{def_mu} and
\begin{equation*}
	\nu_\varphi(A) \coloneqq \mathbb{E}\biggl[\int_0^\infty \bm{1}_A(s, X_s) \int_{\mathbb{R}^d} \varphi(\xi) \,\kappa_s(d\xi) \,ds \biggr],\quad
	A \in \mathcal{B}(\mathbb{R}_+ \times \mathbb{R}^d).
\end{equation*}
For $\varphi \in \mathcal{L}$ with $\varphi \geq 0$, we have that $\nu_\varphi$ is a (positive) measure, so there exists a $\mu$-null set $N_\varphi$ such that for all $(t, x) \notin N_\varphi$,
\begin{equation}\label{lem2_pos}
	L_{t, x}(\varphi) \geq 0.
\end{equation}
For $\varphi, \psi \in \mathcal{L}$ and $p, q \in \mathbb{Q}$, by the uniqueness of the Radon--Nikodym derivative, there exists a $\mu$-null set $N_{\varphi, \psi, p, q}$ such that for all $(t, x) \notin N_{\varphi, \psi, p, q}$,
\begin{equation}\label{lem2_qlin}
	L_{t, x}(p\varphi + q\psi) = pL_{t, x}(\varphi) + qL_{t, x}(\psi).
\end{equation}
We define the $\mu$-null set
\begin{equation*}
	N \coloneqq \Biggl(\bigcup_{\varphi \in \mathcal{L}, \varphi \geq 0} N_\varphi\Biggr) \cup \Biggl(\bigcup_{\varphi, \psi \in \mathcal{L}, p, q \in \mathbb{Q}} N_{\varphi, \psi, p, q}\Biggr).
\end{equation*}
For each $\varphi \in \mathcal{L}$, we modify $L_{t, x}(\varphi) \coloneqq 0$ for $(t, x) \in N$ and keep the same notation. Now by construction, \eqref{lem2_pos} holds for all $(t, x) \in \mathbb{R}_+ \times \mathbb{R}^d$, $\varphi \in \mathcal{L}$ with $\varphi \geq 0$, and \eqref{lem2_qlin} holds for all $(t, x) \in \mathbb{R}_+ \times \mathbb{R}^d$, $\varphi, \psi \in \mathcal{L}$, $p, q \in \mathbb{Q}$. Thus, for fixed $(t, x)$ we see that $L_{t, x}$ is a positive $\mathbb{Q}$-linear functional on $\mathcal{L}$. Moreover, for each $\varphi \in \mathcal{L}$, the function $(t, x) \mapsto L_{t, x}(\varphi)$ is still a version of $\frac{d\nu_\varphi}{d\mu}$, so \eqref{lem2_L} remains true for Lebesgue-a.e.\ $t \geq 0$.

The next step is to extend $L_{t, x}$ to $C_\ell(\mathbb{R}^d)$ for each fixed $(t, x) \in \mathbb{R}_+ \times \mathbb{R}^d$. Let $\varphi \in \mathcal{L}$, and take a sequence $(q_n)_{n \in \mathbb{N}} \subset \mathbb{Q}$ decreasing to $\lVert \varphi \rVert_\infty$. Note that $|\varphi| \leq q_n$ for all $n$, so it follows that
\begin{equation*}
\begin{split}
	L_{t, x}(\varphi)
	= -L_{t, x}(q_n - \varphi) + q_n L_{t, x}(1)
	\leq q_n L_{t, x}(1),\\
	L_{t, x}(\varphi)
	= L_{t, x}(q_n + \varphi) - q_n L_{t, x}(1)
	\geq -q_n L_{t, x}(1),
\end{split}
\end{equation*}
i.e.\ $|L_{t, x}(\varphi)| \leq q_n L_{t, x}(1)$. Letting $n \to \infty$, we obtain that
\begin{equation}\label{lem2_bdd}
	|L_{t, x}(\varphi)| \leq L_{t, x}(1) \lVert \varphi \rVert_\infty.
\end{equation}
By \eqref{lem2_bdd} and the density of $\mathcal{L}$ in $C_\ell(\mathbb{R}^d)$, we can uniquely extend\footnote{
	This extension is based on a standard argument. One delicate point is that $\mathcal{L}$ is a vector space over $\mathbb{Q}$ but $C_\ell(\mathbb{R}^d)$ is a vector space over $\mathbb{R}$. In the proof of the linearity of $L_{t, x}$ on $C_\ell(\mathbb{R}^d)$, we need an extra step simply by the density of $\mathbb{Q}$ in $\mathbb{R}$.
} $L_{t, x}$ to a bounded linear functional on $C_\ell(\mathbb{R}^d)$, and \eqref{lem2_bdd} holds for all $\varphi \in C_\ell(\mathbb{R}^d)$. Moreover, let $\varphi \in C_\ell(\mathbb{R}^d)$ with $\varphi \geq 0$, and take a sequence $(\varphi_n)_{n \in \mathbb{N}} \subset \mathcal{L}$ converging to $\varphi$. Let $0 < \varepsilon \in \mathbb{Q}$. Since $\varphi_n \geq -\varepsilon$ for $n$ large enough and $L_{t, x}$ is positive on $\mathcal{L}$, it follows that
\begin{equation*}
	L_{t, x}(\varphi)
	= \lim_{n \to \infty} L_{t, x}(\varphi_n)
	= \lim_{n \to \infty} L_{t, x}(\varphi_n + \varepsilon) - \varepsilon L_{t, x}(1)
	\geq -\varepsilon L_{t, x}(1).
\end{equation*}
Sending $\varepsilon \to 0$ along rational numbers, we get $L_{t, x}(\varphi) \geq 0$. Thus, $L_{t, x}$ is a positive bounded linear functional on $C_\ell(\mathbb{R}^d)$, and in particular on $C_0(\mathbb{R}^d)$. By the Riesz--Markov--Kakutani representation theorem, there exists a finite (positive) Radon measure, denoted by $k(t, x, d\xi)$, such that
\begin{equation}\label{lem2_riesz}
	L_{t, x}(\varphi)
	= \int_{\mathbb{R}^d} \varphi(\xi) \,k(t, x, d\xi),\quad
	\forall\, \varphi \in C_0(\mathbb{R}^d).
\end{equation}

We claim that $k$ is a finite transition kernel from $\mathbb{R}_+ \times \mathbb{R}^d$ to $\mathbb{R}^d$. For fixed $(t, x) \in \mathbb{R}_+ \times \mathbb{R}^d$, by construction $k(t, x, d\xi)$ is a finite measure. On the other hand, we have that $L_{t, x}(\varphi)$ is measurable in $(t, x)$ for all $\varphi \in \mathcal{L}$, thus for all $\varphi \in C_0(\mathbb{R}^d)$ by pointwise convergence. Since the indicator function of an open cube can be approximated by functions in $C_0(\mathbb{R}^d)$, from \eqref{lem2_riesz} and the monotone convergence theorem we know that $k(t, x, A)$ is measurable in $(t, x)$ for all open cubes $A$. Then by Dynkin's $\pi$-$\lambda$ theorem, measurability holds for all $A \in \mathcal{B}(\mathbb{R}^d)$. This proves our claim.

It only remains to verify \eqref{lem2_rslt} for Lebesgue-a.e.\ $t \geq 0$. The way we argue is similar to the previous paragraph. We already know that for Lebesgue-a.e.\ $t \geq 0$:
\begin{enumerate}[label=(\roman*), nosep]
	\item \eqref{lem2_L} holds for all $\varphi \in \mathcal{L}$, since $\mathcal{L}$ is countable,
	
	\item $\mathbb{E}[\kappa_t(\mathbb{R}^d)] < \infty$, due to \eqref{lem2_asm}.
\end{enumerate}
We fix such ``good'' $t$. Now for $\varphi \in C_0(\mathbb{R}^d)$, take a sequence in $\mathcal{L}$ converging to $\varphi$. By pointwise convergence on the left-hand side and $L^1$-convergence on the right-hand side of \eqref{lem2_L}, it is easy to check that \eqref{lem2_L} holds for all $\varphi \in C_0(\mathbb{R}^d)$. Then by \eqref{lem2_riesz} and the monotone convergence theorem, we know that \eqref{lem2_rslt} holds for all open cubes $A$. Finally, Dynkin's $\pi$-$\lambda$ theorem yields that \eqref{lem2_rslt} holds for all $A \in \mathcal{B}(\mathbb{R}^d)$. This finishes the proof.
\end{proof}

\begin{remark}\label{rem_lem2}
Under the framework of Lemma~\ref{lem2}, with a bit more effort, we can show that for Lebesgue-a.e.\ $t \geq 0$,
\begin{equation}\label{lem2_rslteqv}
	\int_{\mathbb{R}^d} g(X_t, \xi) \,k(t, X_t, d\xi)
	= \mathbb{E}\biggl[\int_{\mathbb{R}^d} g(X_t, \xi) \,\kappa_t(d\xi) \,\bigg|\, X_t\biggr]
\end{equation}
holds for all bounded measurable functions $g: \mathbb{R}^{2d} \to \mathbb{R}$. Indeed, \eqref{lem2_rslt} implies that \eqref{lem2_rslteqv} holds for all $g$ of the form $\bm{1}_{A_1 \times A_2}$ with $A_1, A_2 \in \mathcal{B}(\mathbb{R}^d)$. Dynkin's $\pi$-$\lambda$ theorem then tells us that \eqref{lem2_rslteqv} holds for all $g$ of the form $\bm{1}_E$ with $E \in \mathcal{B}(\mathbb{R}^{2d})$. Finally, a standard approximation argument yields the desired result.
\end{remark}

\subsection{Differential Characteristics}\label{sec_dc}
In this subsection we briefly review the concept of differential characteristics of It\^o semimartingales. For a detailed discussion, the readers can refer to \cite{MR1943877}, Chapter II.2. Note that in this paper, all semimartingales have c\`adl\`ag sample paths by convention.

\begin{definition}
We say $h: \mathbb{R}^d \to \mathbb{R}^d$ is a \emph{truncation function} if $h$ is measurable, bounded and $h(x) = x$ in a neighborhood of $0$.
\end{definition}

Now we give the definition of differential characteristics. Recall that an It\^o semimartingale is a semimartingale whose characteristics are absolutely continuous in the time variable.

\begin{definition}
Let $X = (X^i)_{1 \leq i \leq d}$ be an $\mathbb{R}^d$-valued It\^o semimartingale. The \emph{differential characteristics} of $X$ associated with a truncation function $h$ is the triplet $(\beta, \alpha, \kappa)$ consisting in:
\begin{enumerate}[label=(\roman*), nosep]
\item $\beta = (\beta^i)_{1 \leq i \leq d}$ is an $\mathbb{R}^d$-valued predictable process such that $\int_0^\cdot \beta_s \,ds$ is the predictable finite variation part of the special semimartingale
\begin{equation*}
	X(h)_t = X_t - \sum_{s \leq t} (\Delta X_s - h(\Delta X_s)).
\end{equation*}

\item $\alpha = (\alpha^{ij})_{1 \leq i, j \leq d}$ is an $\mathbb{S}_+^d$-valued predictable process such that
\begin{equation*}
	\int_0^\cdot \alpha_s^{ij} \,ds
	= \langle X^{i, c}, X^{j, c} \rangle,\quad
	1 \leq i, j \leq d,
\end{equation*}
where $X^c = (X^{i, c})_{1 \leq i \leq d}$ is the continuous local martingale part of $X$.

\item $\kappa$ is a predictable L\'evy transition kernel from $\Omega \times \mathbb{R}_+$ to $\mathbb{R}^d$ such that $\kappa_t(d\xi) dt$ is the compensator of the random measure $\mu^X$ associated to the jumps of $X$, namely
\begin{equation*}
	\mu^X(dt, d\xi) = \sum_{s > 0} \bm{1}_{\{\Delta X_s \neq 0\}} \delta_{(s, \Delta X_s)}(dt, d\xi).
\end{equation*}
\end{enumerate}
\end{definition}

\begin{remark}
We require the differential characteristics $(\beta, \alpha, \kappa)$ to be predictable. As was discussed in \cite{MR1943877}, Proposition II.2.9, we can always find such a ``good'' version. We also note that $\alpha$ and $\kappa$ do not depend on the choice of the truncation function $h$, while $\beta = \beta(h)$ does. For two truncation functions $h$, $\widetilde{h}$, the relationship between their corresponding $\beta$ is given by \cite{MR1943877}, Proposition II.2.24:
\begin{equation}\label{betah}
	\beta(h)_t - \beta(\widetilde{h})_t
	= \int_{\mathbb{R}^d} (h(\xi) - \widetilde{h}(\xi)) \,\kappa_t(d\xi).
\end{equation}
\end{remark}

Using differential characteristics, one can write an It\^o semimartingale in its canonical decomposition (\cite{MR1943877}, Theorem II.2.34):
\begin{equation*}
\begin{split}
	X_t = X_0 &+ \int_0^t \beta_s \,ds + X^c_t\\ &+ \int_0^t \int_{\mathbb{R}^d} h(\xi) \,(\mu^X(ds, d\xi) - \kappa_s(d\xi) ds) + \int_0^t \int_{\mathbb{R}^d} (\xi - h(\xi)) \,\mu^X(ds, d\xi).
\end{split}
\end{equation*}
Always perhaps, after enlarging the probability space, we may have the representation $X^c = \int_0^\cdot (\alpha_s)^{1/2} \,dB_s$ for some $d$-dimensional Brownian motion $B$, and this is what we usually see in applications. As our proof does not rely on such It\^o integrals, we stick to the more general setting.

Finally, we give a well-known property of It\^o semimartingales, which will be used in our main results. Since the proof is short, we present it below for completeness.

\begin{proposition}\label{Xt=Xt-}
Let $X$ be an It\^o semimartingale. Then, for each fixed $t \geq 0$, $\Delta X_t = 0$ $\mathbb{P}$-a.s.
\end{proposition}

\begin{proof}
Let $\kappa$ be the third differential characteristic of $X$, i.e.\ $\kappa_s(d\xi) ds$ is the compensator of $\mu^X$. Fix $t \geq 0$, then by the definition of compensators,
\begin{equation*}
	\mathbb{P}(\Delta X_t \neq 0)
	= \mathbb{E}\biggl[\int_{\mathbb{R}_+ \times \mathbb{R}^d} \bm{1}_{\{s=t\}} \,\mu^X(ds, d\xi)\biggr]
	= \mathbb{E}\biggl[\int_{\mathbb{R}_+} \int_{\mathbb{R}^d} \bm{1}_{\{s=t\}} \,\kappa_s(d\xi) \,ds\biggr] = 0.
\end{equation*}
\end{proof}

\section{Main Results}\label{sec3}
In this section we present our main results on Markovian projections for It\^o semimartingales with jumps. Our proof uses the superposition principle for non-local FPKEs established in \cite{MR4168386}. As a consequence, we construct Markovian projections which are solutions to martingale problems, or equivalently, weak solutions to SDEs.

First we recall the notion of martingale problem. Since we are working with semimartingales with jumps, consider the path space $\mathbb{D}(\mathbb{R}_+; \mathbb{R}^d)$ of all c\`adl\`ag functions from $\mathbb{R}_+$ to $\mathbb{R}^d$, endowed with the Skorokhod topology. Let $X$ be the canonical process, i.e.\ $X_t(\omega) = \omega(t)$ for $\omega \in \mathbb{D}(\mathbb{R}_+; \mathbb{R}^d)$ and $t \geq 0$. Let $\mathbb{F}^0$ be the natural filtration generated by $X$, and $\mathbb{F}$ be the right-continuous regularization of $\mathbb{F}^0$. Consider the non-local operator $\mathcal{L} = (\mathcal{L}_t)_{t \geq 0}$ given, for $f \in C^2(\mathbb{R}^d) \cap C_b(\mathbb{R}^d)$ and $x \in \mathbb{R}^d$, by
\begin{equation}\label{FPKO}
\begin{split}
	\mathcal{L}_t f(x) \coloneqq b(t, x) \cdot \nabla f(x) &+ \frac{1}{2} \mathrm{tr}(a(t, x) \nabla^2 f(x))\\ &+ \int_{\mathbb{R}^d} \bigl(f(x + \xi) - f(x) - \nabla f(x) \cdot \xi \bm{1}_{\{|\xi| \leq r\}}\bigr) \,k(t, x, d\xi),
\end{split}
\end{equation}
where $b: \mathbb{R}_+ \times \mathbb{R}^d \to \mathbb{R}^d$, $a: \mathbb{R}_+ \times \mathbb{R}^d \to \mathbb{S}_+^d$ are measurable functions, $k$ is a L\'evy transition kernel from $\mathbb{R}_+ \times \mathbb{R}^d$ to $\mathbb{R}^d$, and $r > 0$ is a constant.

\begin{definition}[Martingale Problem]\label{mtg_pblm}
Let $\mu_0 \in \mathcal{P}(\mathbb{R}^d)$. We call $\widehat{\mathbb{P}} \in \mathcal{P}(\mathbb{D}(\mathbb{R}_+; \mathbb{R}^d))$ a solution to the martingale problem (or a martingale solution) for $\mathcal{L}$ with initial law $\mu_0$, if
\begin{enumerate}[label=(\roman*), nosep]
\item $\widehat{\mathbb{P}} \circ (X_0)^{-1} = \mu_0$,

\item for each $f \in C_c^2(\mathbb{R}^d)$, the process
\begin{equation*}
	M^f_t \coloneqq f(X_t) - f(X_0) - \int_0^t \mathcal{L}_s f(X_s) \,ds
\end{equation*}
is well-defined and an $\mathbb{F}$-martingale under $\widehat{\mathbb{P}}$.
\end{enumerate}
\end{definition}

Under some regularity conditions, e.g.\ local boundedness of $b$, $a$, and $\int_{\mathbb{R}^d} 1 \land |\xi|^2 \,k(\cdot, \cdot, d\xi)$ (which holds under the assumptions of Theorem~\ref{thm_mp}), (ii) in Definition~\ref{mtg_pblm} implies that for each $f \in C^2(\mathbb{R}^d) \cap C_b(\mathbb{R}^d)$, $M^f$ is an $\mathbb{F}$-local martingale under $\widehat{\mathbb{P}}$. In particular, by \cite{MR1943877}, Theorem II.2.42, $X$ admits differential characteristics $b(t, X_{t-})$, $a(t, X_{t-})$ and $k(t, X_{t-}, d\xi)$, associated with the truncation function $h(x) = x \bm{1}_{\{|x| \leq r\}}$. We sometimes also say a process $\widetilde{X}$ is a solution to the martingale problem for $\mathcal{L}$. By this, we mean there exists some filtered probability space and an adapted, c\`adl\`ag process $\widetilde{X}$ on it, such that (i) and (ii) in Definition~\ref{mtg_pblm} are satisfied by $\widetilde{X}$ on its underlying probability space. We can think of it as an analogy to the notion of weak solutions of SDEs.

Now we can state our main results.

\begin{theorem}[Markovian Projection]\label{thm_mp}
Let $X$ be an $\mathbb{R}^d$-valued It\^o semimartingale with differential characteristics $(\beta, \alpha, \kappa)$ associated with the truncation function $h(x) = x \bm{1}_{\{|x| \leq r\}}$ for some $r > 0$. Suppose that $(\beta, \alpha, \kappa)$ satisfies
\begin{equation}\label{mp_asm}
	\mathbb{E}\biggl[\int_0^t \biggl(|\beta_s| + |\alpha_s| + \int_{\mathbb{R}^d} 1 \land |\xi|^2 \,\kappa_s(d\xi)\biggr) \,ds\biggr] < \infty,\quad
	\forall\, t > 0.
\end{equation}
Then, there exist measurable functions $b: \mathbb{R}_+ \times \mathbb{R}^d \to \mathbb{R}^d$, $a: \mathbb{R}_+ \times \mathbb{R}^d \to \mathbb{S}_+^d$, and a L\'evy transition kernel $k$ from $\mathbb{R}_+ \times \mathbb{R}^d$ to $\mathbb{R}^d$ such that for Lebesgue-a.e.\ $t \geq 0$,
\begin{equation}\label{mp_condexp}
	\begin{split}
		b(t, X_{t-}) &= \mathbb{E}[\beta_t \,|\, X_{t-}],\\
		a(t, X_{t-}) &= \mathbb{E}[\alpha_t \,|\, X_{t-}],\\
		\int_A 1 \land |\xi|^2 \,k(t, X_{t-}, d\xi)
		&= \mathbb{E}\biggl[\int_A 1 \land |\xi|^2 \,\kappa_t(d\xi) \,\bigg|\, X_{t-}\biggr],\quad
		\forall\, A \in \mathcal{B}(\mathbb{R}^d).
	\end{split}
\end{equation}
Furthermore, if $(b, a, k)$ satisfies the condition
\begin{equation}\label{grow_cond}
\begin{split}
	\sup_{(t, x) \in \mathbb{R}_+ \times \mathbb{R}^d} \biggl[&\frac{|b(t, x)|}{1 + |x|} + \frac{|a(t, x)|}{1 + |x|^2}\\ &+ \int_{\mathbb{R}^d} \biggl(\bm{1}_{\{|\xi| < r\}} \frac{|\xi|^2}{1 + |x|^2} + \bm{1}_{\{|\xi| \geq r\}} \log\biggl(1 + \frac{|\xi|}{1 + |x|}\biggr)\biggr) \,k(t, x, d\xi)\biggr] < \infty,
\end{split}
\end{equation}
then there exists a solution $\widehat{X}$ to the martingale problem for $\mathcal{L}$, where $\mathcal{L}$ is as defined in \eqref{FPKO}, such that for each $t \geq 0$, the law of $\widehat{X}_t$ agrees with the law of $X_t$.
\end{theorem}

Before proving Theorem~\ref{thm_mp}, we make a few remarks to give more insight into this theorem.

\begin{remark}
Consider the measure $\widetilde{\mu}$ defined as follows:
\begin{equation*}
	\widetilde{\mu}(A)
	\coloneqq \mathbb{E}\biggl[\int_0^\infty \bm{1}_A(s, X_s) \,ds \biggr]
	= \mathbb{E}\biggl[\int_0^\infty \bm{1}_A(s, X_{s-}) \,ds \biggr],\quad
	A \in \mathcal{B}(\mathbb{R}_+ \times \mathbb{R}^d).
\end{equation*}
Intuitively, we can think of $\widetilde{\mu}$ as the ``law'' of $(t, X_t(\omega))$ or $(t, X_{t-}(\omega))$ (though $\widetilde{\mu}$ is not a probability measure). One can easily check that the triplet $(b, a, k(\cdot, \cdot, d\xi))$, which satisfies \eqref{mp_condexp} for Lebesgue-a.e.\ $t \geq 0$, is unique up to a $\widetilde{\mu}$-null set. Moreover, the Markovian projection $\widehat{X}$ is a martingale solution for $\mathcal{L}$, regardless of which version of $(b, a, k)$ is used in \eqref{FPKO}. Indeed, for each $f \in C_c^2(\mathbb{R}^d)$, the function $(t, x) \mapsto \mathcal{L}_t f(x)$ is uniquely defined up to a $\widetilde{\mu}$-null set. We also note that by Fubini's theorem and the mimicking property, $\widetilde{\mu}$ can be written as
\begin{equation*}
	\widetilde{\mu}(A)
	= \widehat{\mathbb{E}}\biggl[\int_0^\infty \bm{1}_A(s, \widehat{X}_s) \,ds \biggr],\quad
	A \in \mathcal{B}(\mathbb{R}_+ \times \mathbb{R}^d),
\end{equation*}
where $\widehat{\mathbb{E}}$ is the expectation on the underlying probability space of $\widehat{X}$. It follows that different versions of $(b, a, k)$ lead to indistinguishable processes $\int_0^\cdot \mathcal{L}_s f(\widehat{X}_s) \,ds$. As a consequence of this observation, condition \eqref{grow_cond} can be weakened by replacing supremum with $\widetilde{\mu}$-essential supremum.
\end{remark}

\begin{remark}
In the theorem we take a truncation function $h(x) = x \bm{1}_{\{|x| \leq r\}}$ for some $r > 0$. Recall that $\beta$ depends on $r$, while $\alpha$, $\kappa$ do not. By \eqref{betah}, we see that the integrability condition \eqref{mp_asm} does not depend on the choice of $r$. However, the growth condition (\ref{grow_cond}) does depend on $r$. One can check that for $0 < r < \widetilde{r}$, if \eqref{grow_cond} holds for $r$, then it also holds for $\widetilde{r}$ (note that $b$ also depends on $r$). The converse is not true in general. In applications, we can pick any specific $r$ such that the assumptions of the theorem are satisfied.
\end{remark}

\begin{remark}
Under \eqref{mp_asm}, one sufficient condition on $X$ that automatically implies \eqref{grow_cond} with $\widetilde{\mu}$-essential supremum is the following: the process
\begin{equation*}
	\frac{|\beta_t|}{1 + |X_t|} + \frac{|\alpha_t|}{1 + |X_t|^2} + \int_{\mathbb{R}^d} \biggl(\bm{1}_{\{|\xi| < r\}} \frac{|\xi|^2}{1 + |X_t|^2} + \bm{1}_{\{|\xi| \geq r\}} \log\biggl(1 + \frac{|\xi|}{1 + |X_t|}\biggr)\biggr) \,\kappa_t(d\xi)
\end{equation*}
(or equivalently replacing $X$ with $X_-$) is bounded up to a $(\mathbb{P} \otimes dt)$-null set. The proof is simply by taking conditional expectations $\mathbb{E}[\cdot \,|\, X_{t-}]$.
\end{remark}

\begin{remark}\label{rmk_mp_cts}
In the case where $X$ is a continuous It\^o semimartingale, i.e.\ $\kappa = 0$, the growth condition \eqref{grow_cond} is not needed. This is exactly Corollary 3.7 (Process itself) in Brunick and Shreve \cite{MR3098443}. We will discuss the continuous case further at the end of this section.
\end{remark}

Now we prove our main theorem.

\begin{proof}[Proof of Theorem~\ref{thm_mp}]
The existence of $b$ and $a$ follows from \eqref{mp_asm} and Lemma~\ref{lem1}, noticing that $\mathbb{S}_+^d$ is a closed convex set in $\mathbb{R}^{d \times d}$. To get the existence of $k$, consider the transition kernel $\widetilde{\kappa}_t(d\xi) \coloneqq 1 \land |\xi|^2 \,\kappa_t(d\xi)$ from $\Omega \times \mathbb{R}_+$ to $\mathbb{R}^d$. \eqref{mp_asm} and Lemma~\ref{lem2} yield a finite transition kernel $\widetilde{k}$ from $\mathbb{R}_+ \times \mathbb{R}^d$ to $\mathbb{R}^d$ such that for Lebesgue-a.e.\ $t \geq 0$,
\begin{equation*}
	\widetilde{k}(t, X_{t-}, A)
	= \mathbb{E}[\widetilde{\kappa}_t(A) \,|\, X_{t-}],\quad
	\forall\, A \in \mathcal{B}(\mathbb{R}^d).
\end{equation*}
For $(t, x) \in \mathbb{R}_+ \times \mathbb{R}^d$, define $k(t, x, d\xi) \coloneqq (1 \land |\xi|^2)^{-1} \widetilde{k}(t, x, d\xi)$ on $\mathbb{R}^d \setminus \{0\}$ and $k(t, x, \{0\}) \coloneqq 0$. Then, $k$ is a L\'evy transition kernel from $\mathbb{R}_+ \times \mathbb{R}^d$ to $\mathbb{R}^d$ that satisfies \eqref{mp_condexp}. Moreover, Remark~\ref{rem_lem2} further tells us that for Lebesgue-a.e.\ $t \geq 0$,
\begin{equation}\label{mp_condexp2}
	\int_{\mathbb{R}^d} g(X_{t-}, \xi) \,k(t, X_{t-}, d\xi)
	= \mathbb{E}\biggl[\int_{\mathbb{R}^d} g(X_{t-}, \xi) \,\kappa_t(d\xi) \,\bigg|\, X_{t-}\biggr]
\end{equation}
holds for all measurable functions $g: \mathbb{R}^{2d} \to \mathbb{R}$ satisfying $|g(x, \xi)| \leq C(1 \land |\xi|^2)$, $\forall\, x, \xi \in \mathbb{R}^d$, for some constant $C > 0$.

Now we prove the second part of Theorem~\ref{thm_mp}. By \cite{MR1943877}, Theorem II.2.42, we know that for each $f \in C_c^2(\mathbb{R}^d)$, the process
\begin{equation*}
\begin{split}
	M_t^f \coloneqq f(X_t) - f(X_0) - \int_0^t \biggl(&\beta_s \cdot \nabla f(X_{s-}) + \frac{1}{2} \mathrm{tr}(\alpha_s \nabla^2 f(X_{s-}))\\ &+ \int_{\mathbb{R}^d} \bigl(f(X_{s-} + \xi) - f(X_{s-}) - \nabla f(X_{s-}) \cdot h(\xi)\bigr) \,\kappa_s(d\xi)\biggr) \,ds
\end{split}
\end{equation*}
is a local martingale. In particular, $M^f$ is locally bounded, thus locally square-integrable and $\langle M^f, M^f \rangle$ is well-defined. We claim that $M^f$ is a (true)  martingale. To show this, it suffices to check $\mathbb{E}[\langle M^f, M^f \rangle_t] < \infty$ for all $t \geq 0$. Let's first compute $[M^f, M^f]$. Note that $M^f - f(X) - f(X_0)$ is a continuous finite variation process, so we have $[M^f, M^f] = [f(X), f(X)]$. By It\^o's formula, the continuous local martingale part of $f(X)$ is given by $\sum_{i=1}^d \int_0^\cdot \partial_i f(X_{s-}) \,dX^{i, c}_s$. Then, it follows from \cite{MR1943877}, Theorem I.4.52 that
\begin{equation*}
\begin{split}
	[f(X), f(X)]_t
	&= \sum_{i=1}^d \sum_{j=1}^d \int_0^t \partial_i f(X_{s-}) \partial_j f(X_{s-}) \,d\langle X^{i, c}, X^{j, c} \rangle_s + \sum_{s \leq t} (f(X_s) - f(X_{s-}))^2\\
	&= \int_0^t \nabla f(X_{s-}) \cdot \alpha_s \nabla f(X_{s-}) \,ds + \int_0^t \int_{\mathbb{R}^d} (f(X_{s-} + \xi) - f(X_{s-}))^2 \,\mu^X(ds, d\xi).
\end{split}
\end{equation*}
Since $\langle M^f, M^f \rangle$ is the compensator of $[M^f, M^f] = [f(X), f(X)]$, we deduce that
\begin{equation*}
\begin{split}
	\langle M^f, M^f \rangle_t
	&= \int_0^t \biggl(\nabla f(X_{s-}) \cdot \alpha_s \nabla f(X_{s-}) + \int_{\mathbb{R}^d} (f(X_{s-} + \xi) - f(X_{s-}))^2 \,\kappa_s(d\xi)\biggr) \,ds\\
	&\leq C\int_0^t \biggl(|\alpha_s| + \int_{\mathbb{R}^d} 1 \land |\xi|^2 \,\kappa_s(d\xi)\biggr) \,ds,
\end{split}
\end{equation*}
where $C = C(\lVert f \rVert_\infty, \lVert \nabla f \rVert_\infty) > 0$ is some constant, and we used the fact that
\begin{equation*}
	|f(x + \xi) - f(x)|^2 \leq C (1 \land |\xi|^2),\quad \forall\, x, \xi \in \mathbb{R}^d.
\end{equation*}
Thus, by \eqref{mp_asm} we get $\mathbb{E}[\langle M^f, M^f \rangle_t] < \infty$ for all $t \geq 0$, which proves our claim that $M^f$ is a martingale.

From the martingale property established above, we have that $\mathbb{E}[M_t^f] = \mathbb{E}[M_0^f] = 0$ for each $t \geq 0$. This allows us to compute
\begin{equation}\label{FPKE_0}
\begin{split}
	&\mathbb{E}[f(X_t)] - \mathbb{E}[f(X_0)]\\
	&\quad\quad= \int_0^t \mathbb{E}\biggl[\beta_s \cdot \nabla f(X_{s-}) + \frac{1}{2} \mathrm{tr}(\alpha_s \nabla^2 f(X_{s-}))\\
	&\quad\quad\quad\quad\quad\,\,\,\,\,
	+ \int_{\mathbb{R}^d} \bigl(f(X_{s-} + \xi) - f(X_{s-}) - \nabla f(X_{s-}) \cdot h(\xi)\bigr) \,\kappa_s(d\xi)\biggr] \,ds\\
	&\quad\quad= \int_0^t \mathbb{E}\biggl[\mathbb{E}[\beta_s \,|\, X_{s-}] \cdot \nabla f(X_{s-}) + \frac{1}{2} \mathrm{tr}(\mathbb{E}[\alpha_s \,|\, X_{s-}] \nabla^2 f(X_{s-}))\\
	&\quad\quad\quad\quad\quad\,\,\,\,\,
	+ \mathbb{E}\biggl[\int_{\mathbb{R}^d} \bigl(f(X_{s-} + \xi) - f(X_{s-}) - \nabla f(X_{s-}) \cdot h(\xi)\bigr) \,\kappa_s(d\xi) \,\bigg|\, X_{s-}\biggr]\biggr] \,ds\\
	&\quad\quad= \int_0^t \mathbb{E}\biggl[b(s, X_{s-}) \cdot \nabla f(X_{s-}) + \frac{1}{2} \mathrm{tr}(a(s, X_{s-}) \nabla^2 f(X_{s-}))\\
	&\quad\quad\quad\quad\quad\,\,\,\,\,
	+ \int_{\mathbb{R}^d} \bigl(f(X_{s-} + \xi) - f(X_{s-}) - \nabla f(X_{s-}) \cdot h(\xi)\bigr) \,k(s, X_{s-}, d\xi)\biggr] \,ds\\
	&\quad\quad= \int_0^t \mathbb{E}[\mathcal{L}_s f(X_{s-})] \,ds,
\end{split}
\end{equation}
where in the first equality Fubini's theorem is justified because of \eqref{mp_asm} and the fact that
\begin{equation}\label{mp_bdd}
	|f(x + \xi) - f(x) - \nabla f(x) \cdot h(\xi)|
	\leq C(1 \land |\xi|^2),\quad \forall\, x, \xi \in \mathbb{R}^d,
\end{equation}
for some constant $C = C(\lVert f \rVert_\infty, \lVert \nabla^2 f \rVert_\infty) > 0$, and in the last but one equality we used \eqref{mp_condexp}, \eqref{mp_condexp2} and \eqref{mp_bdd} once more.

Let $\mu_t$ denote the law of $X_t$. Since $X$ is a c\`adl\`ag process, it is easy to see that the map $\mathbb{R}_+ \ni t \mapsto \mu_t \in \mathcal{P}(\mathbb{R}^d)$ is c\`adl\`ag and $\mu_{t-}$ is the law of $X_{t-}$. Moreover, by Proposition~\ref{Xt=Xt-}, for fixed $t \geq 0$ we have $\Delta X_t = 0$ $\mathbb{P}$-a.s., i.e.\ $X_t = X_{t-}$ $\mathbb{P}$-a.s. This implies that $\mu_t = \mu_{t-}$, and the map $\mathbb{R}_+ \ni t \mapsto \mu_t \in \mathcal{P}(\mathbb{R}^d)$ is actually continuous. Then, \eqref{FPKE_0} can be written as
\begin{equation}\label{FPKE}
	\mu_t(f) = \mu_0(f) + \int_0^t \mu_s(\mathcal{L}_s f) \,ds,\quad
	\forall\, t \geq 0,\, f \in C_c^2(\mathbb{R}^d).
\end{equation}
This shows that $(\mu_t)_{t \geq 0}$ is a weak solution to the non-local FPKE associated with $\mathcal{L}$ in the sense of \cite{MR4168386}, Definition 1.1. Together with the growth condition \eqref{grow_cond}, we are now in a position to apply \cite{MR4168386}, Theorem 1.5.\footnote{
	In the proof of the superposition principle in \cite{MR4168386}, the authors assumed without loss of generality that $r \leq 1/\sqrt{2}$. This is for simplicity in some upper bound estimates, without introducing complicated constants involving $r$. The result actually holds for all $r > 0$.
} We conclude that there exists a solution $\widehat{\mathbb{P}} \in \mathcal{P}(\mathbb{D}(\mathbb{R}_+; \mathbb{R}^d))$ to the martingale problem for $\mathcal{L}$ such that for each $t \geq 0$, the time-$t$ marginal of $\widehat{\mathbb{P}}$ agrees with $\mu_t$. Equivalently, there exists a martingale solution $\widehat{X}$ for $\mathcal{L}$ which mimics the one-dimensional marginal laws of $X$. This finishes the proof.
\end{proof}

As was mentioned in Remark~\ref{rmk_mp_cts}, when $X$ is a continuous It\^o semimartingale, Theorem~\ref{thm_mp} holds without assumption \eqref{grow_cond}. In this case, the setting of the theorem is much simplified: we have $\kappa = 0$, thus $k = 0$. We also do not need the truncation function $h$, so $\beta$ and $b$ have no dependency on $r$. The same type of proof still works here. Indeed, following a similar argument, one can derive the FPKE \eqref{FPKE}. Now $\mathcal{L}$ is a local FPK operator, so we refer to Trevisan \cite{MR3485364}, which implies that the superposition principle holds under the assumption:
\begin{equation*}
	\Gamma_t \coloneqq \int_0^t \int_{\mathbb{R}^d} \bigl(|b(s, x)| + |a(s, x)|\bigr) \,\mu_s(dx) \,ds < \infty,\quad
	\forall\, t \geq 0.
\end{equation*}
This is an immediate consequence of \eqref{mp_asm} and \eqref{mp_condexp}, once we rewrite $\Gamma_t$ in the following way:
\begin{equation*}
	\Gamma_t = \int_0^t \mathbb{E}\bigl[|b(s, X_s)| + |a(s, X_s)|\bigr] \,ds
	\leq \int_0^t \mathbb{E}\bigl[|\beta_s| + |\alpha_s|\bigr] \,ds
	< \infty.
\end{equation*}

For local FPK operators, the superposition principle holds under relatively mild integrability assumptions. However, in the non-local case, the literature is limited and there is no such result to the best of our knowledge. Some boundedness or growth conditions need to be imposed, for example as in \cite{MR4168386}. As of now, assumption \eqref{grow_cond} is needed for general discontinuous It\^o semimartingales. Removing or weakening this assumption is a possible direction of future work.

\section{Examples}\label{sec4}
In applications, Markovian projections usually appear in the inversion problem. More specifically, suppose we start with a relatively simple process $\widehat{X}$. Our goal is to construct a more complicated process $X$, while keeping the one-dimensional marginal laws unchanged. If we manage to find an $X$ such that $\widehat{X}$ is a Markovian projection of $X$, then the marginal law constraints are automatically satisfied. This is what we mean by ``inverting the Markovian projection''. In this section, we present three examples where our Markovian projection theorem can be applied.

\subsection{Local Stochastic Volatility (LSV) Model.}\label{Ex1}
One of the most famous applications of Markovian projections is the calibration of the LSV model in mathematical finance (see \cite{andersen2010interest}, Appendix A, \cite{MR3155635}, Chapter 11, and the references therein). Under the risk-neutral measure, the dynamics of the stock price is modeled via the following SDE (assuming constant interest rate $r$ and no dividend):
\begin{equation}\label{LSV0}
	dS_t = rS_t \,dt + \eta_t \sigma(t, S_t) S_t \,dB_t,
\end{equation}
where $\eta$ is the stochastic volatility, $\sigma$ is a function to be determined, and $B$ is a Brownian motion. Assume that $\eta$ is bounded from above and below by positive constants. One requires the LSV model to be perfectly calibrated to European call option prices (which depends on one-dimensional marginal laws). By the seminal work of Dupire \cite{dupire1994pricing}, we have perfect calibration to European calls in the local volatility (LV) model:
\begin{equation*}
	d\widehat{S}_t = r\widehat{S}_t \,dt + \sigma_{\text{Dup}}(t, \widehat{S}_t) \widehat{S}_t \,d\widehat{B}_t,\quad
	\sigma_{\text{Dup}}^2(t, K) \coloneqq \frac{\partial_t C(t, K) + rK \partial_K C(t, K)}{(1/2) K^2 \partial_{KK} C(t, K)},
\end{equation*}
where $\widehat{B}$ is a Brownian motion, $C(t, K)$ is the European call prices, and we assume that $\sigma_{\text{Dup}}$ is bounded. Thus, it suffices to have $\widehat{S}$ be a Markovian projection of $S$. One can choose
\begin{equation}\label{leverage}
	\sigma(t, x) \coloneqq \frac{\sigma_{\text{Dup}}(t, x)}{\sqrt{\mathbb{E}[\eta_t^2 \,|\, S_t = x]}},
\end{equation}
where the conditional expectation term is understood in the sense of Lemma~\ref{lem1}. Plugging \eqref{leverage} into \eqref{LSV0} yields the McKean--Vlasov type SDE
\begin{equation}\label{LSV}
	dS_t = rS_t \,dt + \frac{\eta_t}{\sqrt{\mathbb{E}[\eta_t^2 \,|\, S_t]}} \sigma_{\text{Dup}}(t, S_t) S_t \,dB_t.
\end{equation}
Suppose \eqref{LSV} admits a solution $S$ starting from $s_0 > 0$. The differential characteristics of $S$ are
\begin{equation*}
	\beta_t = rS_t,\quad
	\alpha_t = \frac{\eta_t^2}{\mathbb{E}[\eta_t^2 \,|\, S_t]} \sigma_{\text{Dup}}^2(t, S_t) S_t^2,\quad
	\kappa_t(d\xi) = 0.
\end{equation*}
By a standard Grönwall type argument, one can show that $S$ is bounded in $L^2$ on any finite time interval $[0, t]$. Thus, assumption \eqref{mp_asm} is satisfied. Taking conditional expectations $\mathbb{E}[\cdot \,|\, S_t]$, we get
\begin{equation*}
	b(t, x) = rx,\quad
	a(t, x) = \sigma_{\text{Dup}}^2(t, x) x^2,\quad
	k(t, x, d\xi) = 0.
\end{equation*}
It then follows from Theorem~\ref{thm_mp} that $\widehat{S}$ is indeed a Markovian projection of $S$.

However, the SDE \eqref{LSV} is notoriously hard to solve, and doing so still remains an open problem in full generality. Partial results exist when $\eta$ is of the form $f(Y)$. For instance, Abergel and Tachet \cite{MR2629564} proved short-time existence of solutions to the corresponding FPKE, with $Y$ being a multi-dimensional diffusion process. Jourdain and Zhou \cite{MR4086600} showed the weak existence when $Y$ is a finite-state jump process and $f$ satisfies a structural condition. Lacker, Shkolnikov and Zhang \cite{MR4152640} showed the strong existence and uniqueness of stationary solutions, when $\sigma_{\text{Dup}}$ does not depend on $t$ and $Y$ solves an independent time homogeneous SDE.

\subsection{Local Stochastic Intensity (LSI) Model.}\label{Ex2}
The LSI model (see \cite{MR3481308}) is a jump process analogue of the LSV model. It is often used in credit risk applications to model the number of defaults via a counting process $X$ whose intensity has the form $\eta_t \lambda(t, X_{t-})$, where $\eta$ is the stochastic intensity and $\lambda$ is a function to be determined. In other words, the process
\begin{equation*}
	X_t - \int_0^t \eta_s \lambda(s, X_{s-}) \,ds
\end{equation*}
is a (local) martingale. Similarly as in Example~\ref{Ex1}, we want the one-dimensional marginal laws of the LSI model to match those of the local intensity (LI) model, which can be perfectly calibrated to collateralized debt obligation (CDO) tranche prices (see \cite{schonbucher2005portfolio}). Note that in the LI model, defaults are modeled via a counting process $\widehat{X}$ whose intensity has the form $\lambda_{\text{Loc}}(t, \widehat{X}_{t-})$.

Assume that $\eta$ is bounded from above and below by positive constants, and $\lambda_{\text{Loc}}$ is bounded. One can choose
\begin{equation*}
	\lambda(t, x) = \frac{\lambda_{\text{Loc}}(t, x)}{\mathbb{E}[\eta_t \,|\, X_{t-} = x]},
\end{equation*}
which yields the McKean--Vlasov type martingale problem:
\begin{equation*}
	\biggl(X_t - \int_0^t \frac{\eta_s}{\mathbb{E}[\eta_s \,|\, X_{s-}]} \lambda_{\text{Loc}}(s, X_{s-}) \,ds\biggr)_{t \geq 0}
	\text{ is a martingale}.
\end{equation*}
The differential characteristics of $X$ are
\begin{equation*}
	\beta_t = 0,\quad
	\alpha_t = 0,\quad
	\kappa_t(d\xi) = \frac{\eta_s}{\mathbb{E}[\eta_s \,|\, X_{s-}]} \lambda_{\text{Loc}}(s, X_{s-}) \delta_1(d\xi),
\end{equation*}
where we used the truncation function $h(x) = x \bm{1}_{\{|x| \leq r\}}$ for $r < 1$. Taking conditional expectations $\mathbb{E}[\cdot \,|\, X_{t-}]$, we get
\begin{equation*}
	b(t, x) = 0,\quad
	a(t, x) = 0,\quad
	k(t, x, d\xi) = \lambda_{\text{Loc}}(t, x) \delta_1(d\xi).
\end{equation*}
Clearly, \eqref{mp_asm} and \eqref{grow_cond} are justified, so it follows from Theorem~\ref{thm_mp} that $\widehat{X}$ is a Markovian projection of $X$. When $\widehat{X}$ is a Poisson process (i.e.\ $\lambda_{\text{Loc}}$ is constant or a deterministic function of time $t$), we call $X$ a fake Poisson process.

Alfonsi, Labart and Lelong \cite{MR3481308} constructed solutions to the LSI model when $\eta_t = f(Y_t)$ for $Y$ either being a discrete state Markov chain or solving an SDE of the following type:
\begin{equation*}
	dY_t = b(t, X_{t-}, Y_{t-}) \,dt + \sigma(t, X_{t-}, Y_{t-}) \,dB_t + \gamma(t, X_{t-}, Y_{t-}) \,dX_t,
\end{equation*}
where $B$ is a Brownian motion. In recent work \cite{LL24}, we prove the existence of solutions to the LSI model under milder regularity conditions, while our $\eta$ is an exogenously given process not in the above feedback form involving $X$. We also extend the jump sizes of $X$ to follow any discrete law with finite first moment.

\subsection{Fake Hawkes Processes.}\label{Ex3}
A Hawkes process $\widehat{X}$ is a self-exciting counting process whose intensity is given by
\begin{equation*}
	\lambda_t = \lambda_0 + \int_0^{t-} K(t-s) \,d\widehat{X}_s
	= \lambda_0 + \sum_{i: \widehat{\tau}_i < t} K(t - \widehat{\tau}_i),
\end{equation*}
where $\lambda_0 > 0$ is the background intensity, $K \in L^1(\mathbb{R}_+; \mathbb{R}_+)$ is the excitation function and $\widehat{\tau}_1 < \widehat{\tau}_2 < \cdots$ are the jump times of $\widehat{X}$. In this example, we consider the most basic excitation function, namely the exponential $K(t) = ce^{-\theta t}$ for some $c, \theta > 0$.

We are interested in inverting the Markovian projection of $\widehat{X}$. However, we observe that the intensity of $\widehat{X}$ depends on the history of $\widehat{X}$. In other words, the differential characteristics of $\widehat{X}$ are not functions of time and the process itself. Therefore, we cannot expect $\widehat{X}$ to be a Markovian projection of some process. To tackle this problem, we lift $\widehat{X}$ to the pair $(\widehat{X}, \widehat{Y})$ by incorporating the right-continuous version of the intensity process, $\widehat{Y} = \lambda_+$, and our goal is to invert the Markovian projection of $(\widehat{X}, \widehat{Y})$.

The specific form of the excitation function allows us to derive the dynamics of $\widehat{Y}$ as
\begin{equation*}
	d\widehat{Y}_t = \theta (\lambda_0 - \widehat{Y}_{t-}) \,dt + c d\widehat{X}_t.
\end{equation*}
We see that the differential characteristics of $(\widehat{X}, \widehat{Y})$ are
\begin{equation*}
	\widehat{\beta}_t = \bigl(0, \theta (\lambda_0 - \widehat{Y}_{t-})\bigr),\quad
	\widehat{\alpha}_t = 0_{2 \times 2},\quad
	\widehat{\kappa}_t(d\xi_1, d\xi_2) = \widehat{Y}_{t-} \delta_{(1, c)}(d\xi_1, d\xi_2),
\end{equation*}
where we used the truncation function $h(x) = x \bm{1}_{\{|x| \leq r\}}$ for $r < \sqrt{1 + c^2}$ (the jump size of $(\widehat{X}, \widehat{Y})$). This inspires us to define $(X, Y)$ as follows: $X$ is a counting process with intensity
\begin{equation*}
	\frac{\eta_t}{\mathbb{E}[\eta_t \,|\, X_{t-}, Y_{t-}]} Y_{t-},
\end{equation*}
and $Y$ satisfies
\begin{equation*}
	Y_t = \lambda_0 + \int_0^t ce^{-\theta (t-s)} \,dX_s
	= \lambda_0 + \sum_{i: \tau_i \leq t} ce^{-\theta (t - \tau_i)},
\end{equation*}
where $\eta$ is some stochastic intensity bounded from above and below by positive constants, and $\tau_1 < \tau_2 < \cdots$ are the jump times of $X$. We can similarly write down the differential characteristics of $(X, Y)$ with the same truncation function:
\begin{equation*}
	\beta_t = \bigl(0, \theta (\lambda_0 - Y_{t-})\bigr),\quad
	\alpha_t = 0_{2 \times 2},\quad
	\kappa_t(d\xi_1, d\xi_2) = \frac{\eta_t}{\mathbb{E}[\eta_t \,|\, X_{t-}, Y_{t-}]} Y_{t-} \delta_{(1, c)}(d\xi_1, d\xi_2).
\end{equation*}
One can show that $(X, Y)$ is bounded in $L^1$ on any finite time interval $[0, t]$. Thus, \eqref{mp_asm} and \eqref{grow_cond} are justified, and Theorem~\ref{thm_mp} tells us that $(X, Y)$ has the same one-dimensional marginal laws as $(\widehat{X}, \widehat{Y})$. We call $(X, Y)$ a fake Hawkes process. In our recent work \cite{LL24}, we prove the existence of such fake Hawkes processes.

\bibliography{bibliography}
\bibliographystyle{abbrv}

\end{document}